\documentclass[11pt,reqno]{amsart}
\usepackage[T1]{fontenc}
\usepackage[utf8]{inputenc}
\usepackage{lmodern}
\usepackage[a4paper,margin=29mm,headheight=14pt]{geometry}
\usepackage{amsmath,amssymb,amsfonts,amsthm,mathtools}
\usepackage{aliascnt}
\usepackage{microtype}
\usepackage[hidelinks,unicode,pdfusetitle]{hyperref}
\usepackage[nameinlink,noabbrev]{cleveref}
\numberwithin{equation}{section}
\hypersetup{pdftitle={Ample vector bundles without Griffiths-semipositive metrics},pdfsubject={Griffiths positivity and ample vector bundles},pdfkeywords={Griffiths conjecture, ample vector bundle, abelian surface, formal adjoint, stable deformations}}
\allowdisplaybreaks[2]
\theoremstyle{plain}
\newtheorem{theorem}{Theorem}[section]
\theoremstyle{plain}
\newaliascnt{proposition}{theorem}
\newtheorem{proposition}[proposition]{Proposition}
\aliascntresetthe{proposition}
\crefname{proposition}{Proposition}{Propositions}
\Crefname{proposition}{Proposition}{Propositions}
\newaliascnt{lemma}{theorem}
\newtheorem{lemma}[lemma]{Lemma}
\aliascntresetthe{lemma}
\crefname{lemma}{Lemma}{Lemmas}
\Crefname{lemma}{Lemma}{Lemmas}
\newaliascnt{corollary}{theorem}
\newtheorem{corollary}[corollary]{Corollary}
\aliascntresetthe{corollary}
\crefname{corollary}{Corollary}{Corollaries}
\Crefname{corollary}{Corollary}{Corollaries}
\theoremstyle{definition}
\newaliascnt{definition}{theorem}

\aliascntresetthe{definition}
\crefname{definition}{Definition}{Definitions}
\Crefname{definition}{Definition}{Definitions}
\theoremstyle{definition}
\newaliascnt{question}{theorem}
\newtheorem{question}[question]{Question}
\aliascntresetthe{question}
\crefname{question}{Question}{Questions}
\Crefname{question}{Question}{Questions}
\theoremstyle{definition}
\newaliascnt{remark}{theorem}
\newtheorem{remark}[remark]{Remark}
\aliascntresetthe{remark}
\crefname{remark}{Remark}{Remarks}
\Crefname{remark}{Remark}{Remarks}
\crefname{theorem}{Theorem}{Theorems}
\crefname{section}{Section}{Sections}

\newcommand{\C}{\mathbb C}
\newcommand{\R}{\mathbb R}
\newcommand{\Z}{\mathbb Z}
\newcommand{\PP}{\mathbb P}
\newcommand{\OO}{\mathcal O}
\newcommand{\HH}{\mathcal H}
\newcommand{\DD}{\mathcal D}
\newcommand{\End}{\operatorname{End}}
\newcommand{\Herm}{\operatorname{Herm}}
\newcommand{\tr}{\operatorname{tr}}
\newcommand{\rk}{\operatorname{rk}}
\newcommand{\Id}{\operatorname{Id}}
\newcommand{\Pic}{\operatorname{Pic}}
\newcommand{\Hom}{\operatorname{Hom}}
\newcommand{\Levi}{\operatorname{Levi}}
\newcommand{\Gr}{\operatorname{Gr}}
\newcommand{\Hilb}{\operatorname{Hilb}}
\newcommand{\norm}[1]{\left\lVert #1\right\rVert}
\newcommand{\ip}[2]{\left\langle #1,#2\right\rangle}
\newcommand{\half}{\tfrac12}
\newcommand{\Hmat}{\mathsf H}
\title[Ample bundles without Griffiths-semipositive metrics]{Ample vector bundles without\\ Griffiths-semipositive metrics}
\author{Yun-Heng Du}
\address{Academy of Mathematics and Systems Science, Chinese Academy of Sciences, Beijing 100190, China}
\email{duyunheng@amss.ac.cn}
\author{Song-Yan Xie}
\address{State Key Laboratory of Mathematical Sciences, Academy of Mathematics and Systems Science, Chinese Academy of Sciences, Beijing 100190, China; School of Mathematical Sciences, University of Chinese Academy of Sciences, Beijing 100049, China}
\email{xiesongyan@amss.ac.cn}
\hypersetup{pdfauthor={Yun-Heng Du; Song-Yan Xie}}
\date{}
\begin{document}
\begin{abstract}
We construct a rank-two ample holomorphic vector bundle on an abelian surface that admits no smooth Griffiths-semipositive Hermitian metric, thereby giving a counterexample to the Griffiths conjecture. The proof combines a curvature obstruction that persists under deformations with the openness of ampleness in an irreducible moduli space of stable sheaves.
\end{abstract}
\maketitle

\section{Introduction}\label{sec:introduction}

Let $E\to X$ be a holomorphic vector bundle on a compact complex manifold.  A smooth Hermitian metric $g$ is \emph{Griffiths positive} if its Chern curvature satisfies
\[
 g\bigl(R^g(u,\bar u)e,e\bigr)>0
 \qquad(0\ne u\in T_x^{1,0}X,\quad 0\ne e\in E_x)
\]
at every $x\in X$.  The metric is \emph{Griffiths semipositive} if the same curvature expression is nonnegative for all $x,u,e$.  Here $T_x^{1,0}X$ is the holomorphic tangent space, and the curvature convention is specified in \cref{eq:chern}.  We use the line convention for projectivisation: $\PP(E)$ parametrises lines in the fibres of $E$.  The bundle $E$ is \emph{ample} if the tautological line bundle $\OO_{\PP(E^*)}(1)$ is ample \cite{Hartshorne}, where $E^*$ is the dual bundle.

Griffiths \cite{Griffiths1969} proved that Griffiths positivity implies ampleness and conjectured the converse.  For line bundles, this equivalence is the celebrated Kodaira embedding theorem \cite{Kodaira1954}.  On curves the conjecture is true: Umemura \cite{Umemura} established the equivalence for vector bundles, and Campana--Flenner \cite[p.~571]{CampanaFlenner} gave a geometric characterisation of ample bundles.  More recently, Murakami \cite{Murakami} gave an analytic proof using nonlinear differential equations.  Related approaches to the metric-existence problem use direct-image metrics \cite{Berndtsson,MourouganeTakayama,Naumann} and nonlinear systems \cite{Demailly2021,Pingali2021,Pingali2023,Mandal}.

In higher rank, a positive metric on the tautological line bundle need not arise from a Hermitian quadratic form on each fibre of $E$.  Finsler geometry makes this distinction explicit \cite{KobayashiFinsler}.  A recent result rules out a universal fibrewise functorial construction of the desired Hermitian metric from an arbitrary prescribed positive tautological metric \cite[Theorem~2.1]{Lempert2026}.  Here functoriality means equivariance under complex-linear isomorphisms of the fibres.  This result does not rule out a construction using global geometry.

We show that the conjecture fails already in rank two on an abelian surface.  Here $X=C\times C$, where $C$ is the elliptic curve with affine equation $y^2=x^3-x$, completed by its point at infinity.

\begin{theorem}\label{thm:main}
There exists an ample holomorphic vector bundle $E$ of rank two on $X$ which admits no smooth Griffiths-semipositive Hermitian metric.  In particular, $E$ admits no smooth Griffiths-positive Hermitian metric.
\end{theorem}

The flexibility of the Oka principle \cite{Forstneric2017} led us to expect that counterexamples might exist.  Since the conjecture holds on curves, surfaces are a natural first case.  On an abelian surface explicit calculations are possible: the universal cover is $\C^2$, the transition maps between lifted coordinates are translations, and constant positive-definite Hermitian forms descend to flat K\"ahler metrics.  We begin with a direct sum of two line bundles $L_1$ and $L_2$, whose curvature can be written explicitly, and then deform its holomorphic structure.

The main new ingredient is a strict curvature obstruction to Griffiths semipositivity that persists under small deformations of the holomorphic structure.  The split bundle $E_0=L_1\oplus L_2$ is constructed with a strict Levi obstruction.  On its dual, evaluating the squared-norm Levi form on horizontal lifts defined by a fixed background connection makes the expression linear in the candidate metric.  A positive weighted sum of tests on pairs of tangent and fibre directions, followed by integration by parts, yields a negative integral for every positive-definite candidate metric.  This contradicts the nonnegativity required by Griffiths semipositivity on $E_0$.  The strict coefficient inequality persists under perturbation of the holomorphic structure.

This local obstruction must be connected with ampleness.  For that we need a stable deformation of the split bundle and an ample bundle in the same irreducible moduli space.  A nonsplit extension of $L_2$ by $L_1$ becomes stable after an arbitrarily small change of polarisation in the appropriate direction.  Multiplying its extension class by a small nonzero scalar gives holomorphic structures arbitrarily close to the split one.  We thus obtain a stable bundle $E_1$ that retains the obstruction.  Local families of stable bundles then give an analytically open neighbourhood $U_{\mathrm{an}}$ of $[E_1]$ in which no represented bundle admits a smooth Griffiths-semipositive Hermitian metric.

For the ample locus, we construct a second stable bundle $E_S$ with the same Chern classes by pulling back an ample stable Steiner bundle on $\PP^2$ along a suitable finite morphism $f:X\to\PP^2$.  The square lattice of $C\simeq\C/(\Z+i\Z)$, preserved by multiplication by $i$, lets us choose the split model to match these Chern classes while retaining the curvature directions needed for the obstruction.

For a common general polarisation, the moduli theorems of Mukai \cite[Theorem~0.1]{Mukai1984} and Yoshioka \cite[Theorem~0.1]{Yoshioka} place $[E_1]$ and $[E_S]$ in a smooth connected, hence irreducible, moduli space.  This irreducibility connects the analytic persistence of the curvature obstruction with the Zariski openness of ampleness in algebraic families \cite[Theorem~1.2.17]{Lazarsfeld}.  The ample point $[E_S]$ makes the Zariski-open ample locus nonempty, hence analytically dense by irreducibility.  Its intersection with the analytic obstruction neighbourhood $U_{\mathrm{an}}$ supplies the required bundle.

Openness of ampleness in families is also used in the proof  \cite{Xie2018} of Debarre's conjecture \cite{Debarre2005}, which asserts that a general complete intersection of $c\ge N/2$ hypersurfaces of sufficiently large degrees in $\PP^N$ has ample cotangent bundle; see also \cite{Brotbek2016,BrotbekDarondeau2018}.  This naturally leads to the question of Griffiths positivity for such cotangent bundles, raised in Section~\ref{sec:cotangent}.

\medskip
\noindent\textbf{Plan of the paper.} In Section~\ref{sec:levi}, we establish the strict Levi obstruction and construct the split bundle $E_0$.  The ample stable comparison bundle $E_S$ is constructed in Section~\ref{sec:comparison}.  Section~\ref{sec:bridge} completes the proof by constructing stable extensions near $E_0$ and placing one of them and $E_S$ in a common irreducible moduli space.  We end with a question in Section~\ref{sec:cotangent}.

\section{The Levi obstruction near a split bundle}\label{sec:levi}

\subsection{The Levi form on horizontal lifts}

Let $F\to X$ be a holomorphic vector bundle of rank $r\ge1$.  The obstruction is formulated for Griffiths-seminegative metrics on $F$; the application to the original problem takes $F=E^*$.  For a smooth Hermitian metric $k$ on $F$, we take inner products to be linear in the first argument and conjugate-linear in the second.  Its squared-norm function on the total space is
\[
 g_k:F\longrightarrow\R,\qquad g_k(z,v)=k_z(v,v).
\]
Choose a local holomorphic frame $(e_1,\ldots,e_r)$ and represent fibre vectors by columns.  The Gram matrix is $G_k=(k(e_a,e_b))$.  For calculations we use its conjugate, denoted by the same letter as the metric:
\[
 k=\overline{G_k}=G_k^t,\qquad
 k_z(v,w)=v^tG_k(z)\bar w=w^{\dagger}k(z)v,\qquad
 g_k(z,v)=v^{\dagger}k(z)v.
\]
Here $B^t$ denotes the transpose of a matrix $B$, and $B^\dagger=\bar B^t$ its conjugate transpose.  The metric adjoint $B^{*k}$ of a bundle endomorphism $B$ is characterised by $k(Bv,w)=k(v,B^{*k}w)$; its matrix is $k^{-1}B^\dagger k$, which reduces to $B^\dagger$ in a $k$-unitary frame.  The same convention for Gram matrices applies to any other Hermitian metric or form.

Write $\partial_i=\partial/\partial z_i$ and $\partial_{\bar j}=\partial/\partial\bar z_j$ in local holomorphic coordinates $z_1,\ldots,z_n$ on $X$.  The Chern connection is the unique connection preserving $k$ whose $(0,1)$-part defines the given holomorphic structure: in a holomorphic frame this part is coefficientwise $\bar\partial$.  For a section represented by a column $s$, it is written $\DD_k s=ds+\Gamma s$, where $\Gamma=\sum_i\Gamma_i\,dz_i$.  Its connection and curvature coefficients are
\begin{equation}\label{eq:chern}
 \Gamma_i=k^{-1}\partial_i k,\qquad
 R_{i\bar j}=-\partial_{\bar j}\Gamma_i.
\end{equation}
The curvature is the endomorphism-valued two-form $R=\DD_k^2$, given by
\[
 R=\bar\partial\Gamma
  =\sum_{i,j}R_{i\bar j}\,dz_i\wedge d\bar z_j.
\]
The minus sign in \eqref{eq:chern} comes from reordering $d\bar z_j\wedge dz_i$.  For a tangent vector $u=\sum_i u_i\partial_i$, the contraction $R(u,\bar u)=\sum_{i,j}u_i\bar u_jR_{i\bar j}$ is an endomorphism of the fibre.  The metric $k$ is Griffiths seminegative when
\[
 k\bigl(R(u,\bar u)v,v\bigr)=v^{\dagger}kR(u,\bar u)v\le0
\]
at every point, for all tangent vectors $u$ and fibre vectors $v$.  Griffiths negativity means strict inequality whenever both vectors are nonzero.  The dual bundle $F^*$ has fibre $\Hom_\C(F_z,\C)$.  Its induced metric $k^*$ has matrix $(k^{-1})^t$ in the dual frame, and its connection and curvature matrices satisfy
\[
 \Gamma^{k^*}=-\Gamma^t,\qquad
 R^{k^*}=-R^t.
\]
Thus a metric on $F$ is Griffiths semipositive if and only if its dual metric on $F^*$ is Griffiths seminegative.

For a smooth real-valued function $g$ in local holomorphic coordinates $(w_1,\ldots,w_N)$ and a tangent vector $\zeta=\sum_a\zeta_a\partial/\partial w_a$, the Levi form is
\[
 \Levi(g)(\zeta)
 =\sum_{a,b=1}^N\frac{\partial^2g}{\partial w_a\,\partial\bar w_b}
       \zeta_a\bar\zeta_b.
\]
A smooth function $g$ is plurisubharmonic precisely when $\Levi(g)(\zeta)\ge0$ for every $\zeta$; it is strictly plurisubharmonic when the inequality is strict for every nonzero $\zeta$.

For any smooth Hermitian metric $h'$ on $F$, write $g_{h'}(z,v)=h'_z(v,v)=v^{\dagger}h'(z)v$ in a local holomorphic frame.  Its Levi form is computed in both the base and fibre variables.  By Forstneri\v c--Kusakabe \cite[Proposition~2.6]{ForstnericKusakabe}, $h'$ is Griffiths seminegative if and only if this function $g_{h'}$ is plurisubharmonic on $F$. If $h'$ is Griffiths negative, then $g_{h'}$ is strictly plurisubharmonic away from the zero section by Drinovec Drnov\v sek--Forstneri\v c \cite[Proposition~6.2(iii)]{DrinovecForstneric2010}.

In a local holomorphic trivialisation, a $(1,0)$ tangent vector to $F$ at $(z,v)$ has the form $(u,w)\in T^{1,0}_{(z,v)}F$, where $u=\sum_i u_i\partial_i\in T_z^{1,0}X$ is the base component and $w\in F_z$ is the fibre-coordinate component.  Let $R^{h'}$ denote the Chern curvature of $h'$.  Subscripts denote directional derivatives, so $h'_u=\sum_i u_i\partial_i h'$ and $h'_{u\bar u}=\sum_{i,j}u_i\bar u_j\partial_i\partial_{\bar j}h'$, and $R^{h'}_{u\bar u}=\sum_{i,j}u_i\bar u_jR^{h'}_{i\bar j}$.  Direct differentiation of the squared norm gives
\[
 \Levi(g_{h'})(u,w)
 =w^{\dagger}h'w+2\operatorname{Re}(w^{\dagger}h'_u v)+v^{\dagger}h'_{u\bar u}v.
\]
Completing the square in $w$ (see \cite[\S1.1]{ZhangSchur}) and using \eqref{eq:chern} gives the Schur-complement identity
\begin{equation}\label{eq:levi-block}
 \Levi(g_{h'})(u,w)=\bigl(w+(h')^{-1}h'_u v\bigr)^\dagger h'\bigl(w+(h')^{-1}h'_u v\bigr)-v^{\dagger}h'R^{h'}_{u\bar u}v.
\end{equation}
The converse in the strict case follows from \eqref{eq:levi-block}: if $g_{h'}$ is strictly plurisubharmonic away from the zero section, taking $w=-(h')^{-1}h'_u v$ gives $-v^{\dagger}h'R^{h'}_{u\bar u}v>0$ for nonzero $u$ and $v$, so $h'$ is Griffiths negative.

We now keep the underlying smooth bundle $F$ fixed and allow its holomorphic structure to vary.  Write $\Omega^{p,q}(F)$ for smooth $F$-valued forms of type $(p,q)$.  By the Koszul--Malgrange theorem \cite[Proposition~1.3.7]{KobayashiDG}, a holomorphic structure $J$ on $F$ is equivalently specified by an integrable Dolbeault operator, that is, a $\C$-linear first-order differential operator satisfying
\[
 \bar\partial_J:\Omega^{0,0}(F)\longrightarrow\Omega^{0,1}(F),
 \qquad
 \bar\partial_J(fs)=\bar\partial f\otimes s+f\bar\partial_Js,
 \qquad \bar\partial_J^2=0.
\]
Here $f$ is a smooth function and $s$ a smooth section.  The operator extends to bundle-valued forms by the graded Leibniz rule, and $\bar\partial_J^2=0$ is the integrability condition.  The local holomorphic sections are those satisfying $\bar\partial_Js=0$.

Fix a smooth Hermitian metric $k$ on $F$, and let $\DD_J$ be the Chern connection of $(\bar\partial_J,k)$.  This connection stays fixed while the candidate metric $h'$ varies.  At an arbitrary point $(z,v)$ of the total space of $F$, with $v\in F_z$, use a local $J$-holomorphic trivialisation to write a $(1,0)$ tangent vector to $F$ as $(u,w)\in T^{1,0}_{(z,v)}F$, where $u\in T_z^{1,0}X$ and $w\in F_z$.  Write $\Gamma_u=\sum_i u_i\Gamma_i$.  The vector $(u,w)$ is horizontal when the covariant change of the fibre vector in this direction is zero, that is, $w+\Gamma_u v=0$.  With $(z,v)$ fixed, each base tangent vector $u$ therefore has the unique horizontal lift
\begin{equation}\label{eq:lift}
 u^{\mathrm h,J}_v=(u,-\Gamma_u v).
\end{equation}
For this definition, $h'$ may be any smooth Hermitian form, without a positivity assumption; we still set $g_{h'}(z,v)=h'_z(v,v)$.  Write $\Levi_J$ for the Levi form in the holomorphic total space defined by $J$.  The \emph{horizontal Levi form} of $h'$ is
\begin{equation}\label{eq:horizontal}
 \HH_J(h')(u,v):=\Levi_J(g_{h'})(u^{\mathrm h,J}_v).
\end{equation}
If $h'$ is a Griffiths-seminegative metric, plurisubharmonicity gives $\HH_J(h')(u,v)\ge0$ for all $u,v$.  The testing directions are independent of $h'$, so $\HH_J$ is a real-linear second-order differential operator in $h'$.  This operator gives a linear necessary condition on the unknown metric, which becomes an integral obstruction in \cref{sec:adjoint}.

Differentiating $g_{h'}(z,v)=v^{\dagger}h'(z)v$ and substituting $w=-\Gamma_u v$ gives
\begin{equation}\label{eq:Hdirect}
 \HH_J(h')(u,v)=v^{\dagger}\left(
 h'_{u\bar u}-\Gamma_u^{\dagger}h'_u-h'_{\bar u}\Gamma_u
                  +\Gamma_u^{\dagger}h'\Gamma_u\right)v.
\end{equation}
This expression is Hermitian quadratic in each of $u$ and $v$.

For the remainder of this section, fix a Kähler metric on $X$; in the application it is the flat metric used in \cref{sec:split}.  Let $\nabla^X$ be its Chern connection, which equals the Levi--Civita connection.

The fixed metric $k$ identifies the varying Hermitian form $h'$ with the unique $k$-self-adjoint endomorphism $A:F\to F$ satisfying
\[
 h'_z(v,w)=k_z(A_zv,w)\qquad(z\in X,\ v,w\in F_z).
\]
In the matrix convention above, $A=k^{-1}h'$.  Conversely, $kA$ denotes the Hermitian form $(v,w)\mapsto k(Av,w)$.

The induced connection $\nabla^J$ on the endomorphism bundle $\End(F)$ differentiates $A$ by the rule $(\nabla^J A)s=\DD_J(As)-A\DD_Js$.  To differentiate once more, we also use $\nabla^X$ on the tangent argument.  Thus the covariant Hessian is, for example,
\[
 ((\nabla^J)^2A)(U,\bar V)
 =\nabla^J_U\nabla^J_{\bar V}A
  -\nabla^J_{\nabla^X_U\bar V}A.
\]
For a $k$-self-adjoint $A$, let $\mathcal L_J(A)$ be the $\End(F)$-valued Hermitian form on $T^{1,0}X$ characterised by
\[
 k\bigl(\mathcal L_J(A)(u,\bar u)v,v\bigr)=\HH_J(kA)(u,v).
\]
Thus $\mathcal L_J$ is the real-linear differential operator obtained by writing $\HH_J$ in terms of endomorphisms rather than Hermitian forms.  At a point we use Kähler-normal holomorphic coordinates on $X$ and a local holomorphic frame of $F$ with $k=I$ and $\partial k=0$ at the point, so the relevant first-order connection coefficients vanish there.

\begin{proposition}\label{prop:covariant}
With the convention in \eqref{eq:chern}, the endomorphism-valued Hermitian form $\mathcal L_J(A)$ is
\begin{equation}\label{eq:Hcovariant}
 \mathcal L_J(A)_{i\bar j}
 =\half\left(((\nabla^J)^2A)_{i\bar j}
             +((\nabla^J)^2A)_{\bar j i}\right)
  -\half\{R_{i\bar j},A\},
\end{equation}
where $R$ is the curvature of $\DD_J$ and $\{R,A\}=RA+AR$ is the anticommutator.
\end{proposition}

\begin{proof}
At the point under consideration, $k=I$, $\partial_i k=0$, and
\[
 k_{i\bar j}=-R_{i\bar j}.
\]
Differentiate $A=k^{-1}h'$ twice. The first derivatives of $k$ vanish, hence
\[
 A_{i\bar j}=h'_{i\bar j}+R_{i\bar j}A.
\]
Write $[R,A]=RA-AR$ for the commutator.  For the induced connection on endomorphisms and the Kähler-normal coordinates,
\[
 ((\nabla^J)^2A)_{i\bar j}=A_{i\bar j},\qquad
 ((\nabla^J)^2A)_{\bar j i}=A_{i\bar j}-[R_{i\bar j},A].
\]
The symmetrised expression on the right of \cref{eq:Hcovariant} therefore equals $h'_{i\bar j}$. This is precisely the horizontal Levi coefficient in the normal frame, by \cref{eq:Hdirect}. Both descriptions are tensorial, giving the claimed identity.
\end{proof}

For a line bundle with local metric $k=e^{-\phi}$ and Hermitian form $h'=ae^{-\phi}$, where $\phi,a$ are smooth real functions, formula \eqref{eq:Hdirect} becomes
\begin{equation}\label{eq:scalar}
 \Levi(h'|v|^2)(u,\phi_u v)
 =e^{-\phi}|v|^2\left(a_{u\bar u}-a\phi_{u\bar u}\right).
\end{equation}
All terms involving first derivatives of $\phi$ cancel.  Here $R=\partial\bar\partial\phi$, as in \cite[Ch.~V, (12.8)]{DemaillyCADG}, so the expression is $e^{-\phi}|v|^2(a_{u\bar u}-aR_{u\bar u})$.

\subsection{The adjoint obstruction and persistence}\label{sec:adjoint}

Every Griffiths-seminegative metric $h'$ satisfies $\HH_J(h')(u,v)\ge0$, so any nonnegative weighted sum of these values is also nonnegative.  A separable tensor records the weights and directions of such a sum.  The Hermitian metrics on $T^{1,0}X$ and $F$ remain fixed.  For a Hermitian vector space $W$, let $\Herm(W)$ denote the real vector space of self-adjoint endomorphisms of $W$.  The symbols $\tr$ and $\Id_F$ denote the trace and the identity endomorphism of $F$, respectively.  A tensor
\[
 Q_x\in \Herm(T_x^{1,0}X)\otimes_{\R}\Herm(F_x)
\]
acts as a Hermitian endomorphism of $T_x^{1,0}X\otimes F_x$.  It is called \emph{separable} if it belongs to the closed convex cone
\[
 \operatorname{Sep}_x
 =\overline{\operatorname{cone}}\{P_u\otimes P_v:u\in T_x^{1,0}X,\ v\in F_x\}.
\]
Here the endomorphisms $P_u$ and $P_v$ act by
\[
 P_u(w)=\langle w,u\rangle u\quad(w\in T_x^{1,0}X),\qquad
 P_v(e)=k(e,v)v\quad(e\in F_x).
\]
These positive semidefinite endomorphisms are represented by $uu^{\dagger}$ and $vv^{\dagger}$ in unitary frames.  For nonzero $u$, $P_u$ has rank one and is an orthogonal projection precisely when $\langle u,u\rangle=1$.  Similarly, for nonzero $v$, $P_v$ has rank one and is an orthogonal projection precisely when $k(v,v)=1$.  The notation $\operatorname{cone}$ means the set of finite nonnegative linear combinations.  The tensor used below is separable by an explicit decomposition into tensor products of positive semidefinite tangent endomorphisms and orthogonal line projections.

If $\mathcal B(u,v)$ is Hermitian quadratic in $u$ and $v$, its coefficients in unitary frames give a unique Hermitian matrix $B$ on $T_x^{1,0}X\otimes F_x$ satisfying
\[
 \mathcal B(u,v)=(u\otimes v)^{\dagger}B(u\otimes v).
\]
The trace pairing
\[
 \ip{Q}{\mathcal B}=\tr(QB)
\]
is independent of the unitary frames and satisfies
\[
 \ip{P_u\otimes P_v}{\mathcal B}
 =\tr\bigl((P_u\otimes P_v)B\bigr)=\mathcal B(u,v).
\]
It therefore gives a pairing independent of the decomposition of $Q$.  By linearity and continuity, $\ip{Q}{\mathcal B}\ge0$ for every separable $Q$ whenever $\mathcal B(u,v)\ge0$ for all $u,v$.

Assume that $X$ is compact without boundary.  Fix a smooth separable Hermitian tensor $Q$ and a positive smooth volume form $dV$ on $X$.  For a holomorphic structure $J$, the associated functional is
\begin{equation}\label{eq:I}
 I_J(h')=\int_X\ip{Q}{\HH_J(h')}\,dV.
\end{equation}
Write $h'=kA$ as above.  Since $\mathcal L_J$ is a real-linear differential operator, integration by parts moves all derivatives of $A$ onto the smooth coefficients and $Q$.  The result is a smooth $k$-self-adjoint endomorphism $S_J$, characterised by
\[
 I_J(h')=\int_X\tr(S_JA)\,dV,\qquad A=k^{-1}h'.
\]
Thus $S_J=\mathcal L_J^*Q$: the formal adjoint uses $dV$, the real-bilinear trace pairing $(S,A)\mapsto\tr(SA)$, and the tensor pairing above.  Uniqueness follows by testing arbitrary smooth Hermitian endomorphisms supported in a coordinate chart.  Thus the candidate metric enters only through $A$, without derivatives.  A uniformly negative-definite $S_J$ makes $I_J(h')<0$ for every positive-definite $A$, contradicting the nonnegativity required of a Griffiths-seminegative metric.

\begin{proposition}\label{crit:fixed}
If $S_J(x)\le-\eta\Id_F$ for every $x\in X$, for some $\eta>0$, then $(F,J)$ admits no smooth Griffiths-seminegative Hermitian metric. Equivalently, its dual admits no smooth Griffiths-semipositive Hermitian metric.
\end{proposition}

\begin{proof}
For such a metric $h'$, the nonnegativity of the horizontal Levi form in \eqref{eq:Hdirect} and separability imply $I_J(h')\ge0$. On the other hand, $A=k^{-1}h'$ is positive definite and
\[
 \tr(S_JA)\le-\eta\tr A.
\]
To justify this inequality without commuting $S_J$ and $A$, take the trace of the positive semidefinite matrix $A^{1/2}(-S_J-\eta I)A^{1/2}$. Integrating yields $I_J(h')<0$, a contradiction.
\end{proof}

To use this obstruction beyond the split bundle, its strict negative bound must survive perturbations of the holomorphic structure.  We therefore keep $k$, $Q$, and $dV$ fixed and compare the operators in \cref{eq:horizontal} for nearby $J$ on the same smooth bundle.  Relative to the operator $\bar\partial_{J_0}$ of a fixed reference holomorphic structure $J_0$, the operator of a nearby holomorphic structure $J$ has the form
\[
 \bar\partial_J=\bar\partial_{J_0}+\alpha,
 \qquad \alpha\in\Omega^{0,1}(\End F).
\]
All $C^m$-neighbourhoods and convergence statements for holomorphic structures below refer to their associated Dolbeault operators on the fixed smooth bundle.  The topology is induced by the $C^m$ coefficient norms of $\alpha=\bar\partial_J-\bar\partial_{J_0}$, computed in a fixed finite system of smooth $k$-unitary local trivialisations. Equivalently, these norms may be computed using a fixed reference connection.  With $k$ fixed, the difference of the corresponding Chern connections is $b=\alpha-\alpha^{*k}$, where $\alpha^{*k}$ combines the $k$-adjoint on the coefficients with complex conjugation of the differential-form part.

\begin{proposition}\label{prop:persistence}
Let $J_0$ be a holomorphic structure and write the Chern connection of another holomorphic structure $J$ as $\DD_J=\DD_{J_0}+b$.  There are constants $C,\varepsilon_0>0$, depending on $J_0$, $Q$, the fixed Hermitian metrics on $T^{1,0}X$ and $F$, the volume form $dV$, and the chosen norms, such that
\begin{equation}\label{eq:continuity}
 \norm{S_J-S_{J_0}}_{C^0}
 \le C\left(\norm b_{C^1}+\norm b_{C^0}^2\right)
\end{equation}
whenever $\norm b_{C^1}<\varepsilon_0$.  In particular, a strict negative upper bound for $S_{J_0}$ persists under sufficiently small deformations through holomorphic structures.
\end{proposition}

\begin{proof}
Use a finite system of fixed smooth $k$-unitary trivialisations and a partition of unity. The $(0,1)$-part of a unitary connection determines its $(1,0)$-part, so $C^m$ convergence of Dolbeault operators gives the corresponding convergence of their Chern connections with $k$ fixed.

The Chern connection $\DD_{J_0}$ (resp.\ $\DD_J$) induces the connection $\nabla^0$ (resp.\ $\nabla^J$) on $\End(F)$ and has curvature $R^0$ (resp.\ $R^J$) on $F$.  All second derivatives below use the product connection with $\nabla^X$, and the displayed coefficient calculation is made at the centre of Kähler-normal coordinates.  Write
\[
 \Delta_{i\bar j}A
 :=\half\Bigl((\nabla^J_i\nabla^J_{\bar j}+\nabla^J_{\bar j}\nabla^J_i)
 -(\nabla^0_i\nabla^0_{\bar j}+\nabla^0_{\bar j}\nabla^0_i)\Bigr)A.
\]
Then
\[
 \Delta_{i\bar j}A=[b_i,\nabla^0_{\bar j}A]+[b_{\bar j},\nabla^0_iA]+\half[\nabla^0_i b_{\bar j}+\nabla^0_{\bar j}b_i,A]+\half\bigl([b_i,[b_{\bar j},A]]+[b_{\bar j},[b_i,A]]\bigr).
\]
The curvature difference is
\[
 R^J_{i\bar j}-R^0_{i\bar j}
 =\nabla^0_i b_{\bar j}-\nabla^0_{\bar j}b_i+[b_i,b_{\bar j}].
\]
It contributes the corresponding anticommutator in \cref{eq:Hcovariant}. There is no second derivative of $A$ in the difference of the operators. The coefficients of its first derivatives are linear in $b$, and the zero-order coefficients are linear in $\nabla^0b$ or quadratic in $b$.

Pair with the fixed smooth $Q$ in \cref{eq:I}. On a coordinate neighbourhood $U\subset X$ with real coordinates $x=(x_1,\ldots,x_{2n})$, write $dV=\rho\,dx$. After multiplying by the corresponding function in the partition of unity, the local coefficients $B_j$ and $B_0$ have compact support in $U$.  For each real component $a$ of $A$ in a fixed orthonormal frame of $\Herm(F)$, integration by parts therefore gives
\[
 \int_U\Bigl(\sum_{j=1}^{2n}B_j\frac{\partial a}{\partial x_j}+B_0a\Bigr)\rho\,dx
 =\int_U\Bigl(B_0-\rho^{-1}\sum_{j=1}^{2n}\frac{\partial(\rho B_j)}{\partial x_j}\Bigr)a\rho\,dx.
\]
There is no boundary term because the coefficients have compact support in $U$.  Each $B_j$ for $j\ge1$ is linear in $b$ with fixed smooth coefficients, while $B_0=O(\nabla^0b)+O(b)+O(b^2)$. Thus the displayed coefficient involves at most first derivatives of $b$ and is bounded by the right side of \cref{eq:continuity}. Constants incorporate the finite coordinate cover, fixed coefficients, and their derivatives. The resulting coefficient $S_J-S_{J_0}$ is $k$-self-adjoint because the difference of the original functionals is real on $k$-self-adjoint tests. Summing over components and coordinate charts proves \cref{eq:continuity}.
\end{proof}

\subsection{The split bundle and the formal adjoint}\label{sec:split}

On the square torus, the curvature and first Chern class of each line bundle in the split construction are represented by a constant Hermitian matrix.  We choose these matrices so that the split bundle satisfies the strict adjoint obstruction and has the same Chern classes as the ample comparison bundle.  Write $c_i(G)$ for the $i$th Chern class of a vector bundle $G$, and use a dot for the intersection pairing on $X$.

The uniformisation of $C$ and its order-four automorphism \cite[Corollary~VI.5.1.1 and Example~III.4.4]{Silverman} identify $X$ analytically with $\C^2/\Lambda$, where $\Lambda=(\Z+i\Z)^2$. Write $z=(z_1,z_2)^t$ for the standard coordinates on $\C^2$. If $M$ is Hermitian with integral diagonal entries and Gaussian-integer off-diagonal entries, the alternating form
\[
 E_M(\gamma,\delta)=\operatorname{Im}(\bar\gamma^{\,t}M\delta),\qquad\gamma,\delta\in\Lambda
\]
is integer-valued on $\Lambda$.  In terms of the first-variable-linear Hermitian form $\mathcal H_M(u,v)=v^{\dagger}Mu$, this is $E_M(\gamma,\delta)=-\operatorname{Im}\mathcal H_M(\gamma,\delta)$.  We next construct a line bundle from $M$ and describe its first Chern class.  The Appell--Humbert theorem \cite[Theorem~2.2.3]{BirkenhakeLange} associates a holomorphic line bundle $L_M$ to $\mathcal H_M$, and the integral of $c_1(L_M)$ over the oriented lattice parallelogram with ordered edge vectors $\gamma,\delta\in\Lambda$ equals $E_M(\gamma,\delta)$.  Birkenhake--Lange use the alternating form $\operatorname{Im}\mathcal H_M=-E_M$; their conversion to a de Rham representative is given in \cite[Lemma~3.6.4]{BirkenhakeLange}. The translation-invariant $(1,1)$-form representing $c_1(L_M)$ is
\[
 \omega_M=\frac{i}{2}\,dz^tM^t\,d\bar z.
\]

To make the Appell--Humbert trivialisation explicit, choose a semicharacter $\chi_M:\Lambda\to\{w\in\C:|w|=1\}$ satisfying
\[
 \chi_M(\gamma+\delta)=\chi_M(\gamma)\chi_M(\delta)\exp\bigl(\pi iE_M(\gamma,\delta)\bigr),
 \qquad \gamma,\delta\in\Lambda,
\]
and set
\[
 j_\gamma(z)=\chi_M(\gamma)
 \exp\left(\pi\mathcal H_M(z,\gamma)+\frac{\pi}{2}\mathcal H_M(\gamma,\gamma)\right).
\]
Since $\mathcal H_M(\delta,\gamma)-\mathcal H_M(\gamma,\delta)=2iE_M(\gamma,\delta)$, these holomorphic multipliers satisfy
\[
 \frac{j_\gamma(z+\delta)j_\delta(z)}{j_{\gamma+\delta}(z)}
 =\exp\left(-\pi iE_M(\gamma,\delta)
       +\frac{\pi}{2}\bigl(\mathcal H_M(\delta,\gamma)-\mathcal H_M(\gamma,\delta)\bigr)\right)=1.
\]
Thus $(z,\xi)\mapsto(z+\gamma,j_\gamma(z)\xi)$ defines a lattice action. We take its quotient as our choice of $L_M$.

In the Appell--Humbert trivialisation on $\C^2$, with fibre coordinate $\xi\in\C$, the metric is $\|\xi\|_z^2=|\xi|^2\exp(-\pi z^{\dagger}Mz)$ \cite[\S3.4, p.~57, equation~(3)]{BirkenhakeLange}. The identity
\[
 |j_\gamma(z)|^2\exp\bigl(-\pi\mathcal H_M(z+\gamma,z+\gamma)\bigr)
 =\exp\bigl(-\pi\mathcal H_M(z,z)\bigr)
\]
shows that this metric descends to $L_M$. Its Chern curvature $R^{L_M}=\pi\,dz^tM^t\,d\bar z$ satisfies $\frac{i}{2\pi}R^{L_M}=\omega_M$.  We use these constant-curvature metrics; only their curvature forms, not the metrics themselves, are asserted to be translation-invariant. With $z_j=x_j+iy_j$, the wedge product is
\[
 \omega_M\wedge\omega_M
 =2\det(M)\,dx_1\wedge dy_1\wedge dx_2\wedge dy_2.
\]
The square-lattice fundamental domain has volume $1$ in these real coordinates, so
\[
 \int_Xc_1(L_M)^2=2\det M,
\]
and, for any two such matrices $M,N$, polarisation gives
\begin{equation}\label{eq:mixed-intersection}
 c_1(L_M)\cdot c_1(L_N)=\det(M+N)-\det M-\det N.
\end{equation}

Fix the polarisation
\[
 H=\OO_C(4[0])\boxtimes\OO_C(5[0]),\qquad
 h=c_1(H),\qquad h^2=\int_X h\smile h=40.
\]
Here $0\in C$ is the point at infinity, $[0]$ is its divisor, and $\OO_C(m[0])$ is the associated line bundle.  For the projections $p_1,p_2:X=C\times C\to C$ onto the first and second factors, the exterior tensor product is $L\boxtimes M=p_1^*L\otimes p_2^*M$ for line bundles $L,M$ on $C$; thus $H=p_1^*\OO_C(4[0])\otimes p_2^*\OO_C(5[0])$.  The line bundles $\OO_C(4[0])$ and $\OO_C(5[0])$ have degrees $4$ and $5$, respectively, so both are very ample on the elliptic curve $C$; hence their exterior tensor product $H$ is very ample.

Under the Appell--Humbert identification, the polarisation class $h=c_1(H)$ is represented by
\[
 \Hmat=\begin{pmatrix}4&0\\0&5\end{pmatrix}.
\]
The $H$-slope of a torsion-free sheaf $\mathcal F$ of positive rank is
\[
 \mu_H(\mathcal F)=\frac{c_1(\mathcal F)\cdot h}{\rk\mathcal F}.
\]
Slope stability requires $\mu_H(\mathcal G)<\mu_H(\mathcal F)$ for every coherent subsheaf $\mathcal G\subset\mathcal F$ with $0<\rk\mathcal G<\rk\mathcal F$; semistability uses $\le$.

We choose two Hermitian matrices with integral diagonal and Gaussian-integer off-diagonal entries whose sum is $5\Hmat$ and whose associated first Chern classes have equal intersection with $h$, so that the resulting split bundle has determinant $H^{\otimes5}$ and equal-slope line factors.  A convenient choice is
\begin{equation}\label{eq:split-matrices}
 A_1=\begin{pmatrix}8&1+13i\\1-13i&15\end{pmatrix},\qquad
 A_2=\begin{pmatrix}12&-1-13i\\-1+13i&10\end{pmatrix}.
\end{equation}
Choose a line bundle $L_1$ whose first Chern class is represented by $A_1$ under this identification, and take
\[
 L_2:=H^{\otimes5}\otimes L_1^{-1},\qquad
 E_0=L_1\oplus L_2,\qquad F_0=E_0^*.
\]
Then $c_1(L_2)$ is represented by $A_2$ and $\det E_0=H^{\otimes5}$ as line bundles. By the $\partial\bar\partial$ lemma \cite[Ch.~VI, Lemma~(8.6)]{DemaillyCADG}, a Hermitian metric on each of the prescribed line bundles $H,L_1$ can be adjusted so that its Chern curvature is the translation-invariant representative specified above.  Equip $L_2=H^{\otimes5}\otimes L_1^{-1}$ with the induced tensor-product metric.  Its curvature is the form associated with $A_2=5\Hmat-A_1$, and the determinant metric is the fifth tensor power of the chosen metric on $H$.

Direct multiplication gives
\[
 A_1+A_2=5\Hmat,\qquad \det A_1=\det A_2=-50,\qquad \det(A_1-A_2)=-700.
\]
For $s=1,2$, since $\det\Hmat=20$ and $\tr(\Hmat^{-1}A_s)=5$, \cref{eq:mixed-intersection} gives
\[
 c_1(L_s)\cdot h=20\cdot5=100.
\]
Thus the two line bundles have the same $H$-slope. They also have the same Euler characteristic, namely $\chi(L_s)=c_1(L_s)^2/2=-50$, by Riemann--Roch for line bundles on a complex torus \cite[Theorem~3.6.3]{BirkenhakeLange}.

Since $\int_Xc_1(L_s)^2=2\det A_s=-100$ and $c_1(E_0)=5h$, the Chern numbers and Euler characteristic are
\[
 c_1(E_0)=5h,\qquad \int_Xc_1(E_0)^2=1000,\qquad
 \int_Xc_2(E_0)=600,\qquad \chi(E_0)=-100.
\]
For the later moduli-space argument, the rank and Chern classes are recorded together in the Mukai vector.  Since $X$ has Todd class $1$, it is
\[
 v(\mathcal F)=\operatorname{ch}(\mathcal F)=(r,c,s),
 \qquad r=\rk\mathcal F,\quad c=c_1(\mathcal F),\quad s=\chi(\mathcal F),
\]
where $\operatorname{ch}$ is the Chern character \cite[p.~101]{Mukai1984}.  We use the Mukai pairing
\[
 \langle(r,c,s),(r',c',s')\rangle=c\cdot c'-rs'-r's,
 \qquad v^2=\langle v,v\rangle.
\]
This is also the convention used by Yoshioka \cite[\S1]{Yoshioka}.  The Chern classes computed above give the Mukai vector of $E_0$ and its square:
\[
 v(E_0)=(2,5h,-100),\qquad v(E_0)^2=(5h)^2-2\cdot2\cdot(-100)=1400.
\]
Write
\[
 \mathbf v:=v(E_0)=(2,5h,-100).
\]
This Mukai vector is primitive: any nontrivial common divisor must divide the rank $2$, whereas $5h$ has degree $25$ on the second coordinate elliptic curve and is therefore not divisible by $2$ in integral cohomology.

For an integer $m\ge1$ and a locally free rank-two cokernel $S_m$ in a Steiner resolution
\[
 0\longrightarrow\OO_{\PP^2}(-1)^m
 \longrightarrow\OO_{\PP^2}^{m+2}
 \longrightarrow S_m\longrightarrow0,
\]
one has $c_1(S_m)=m\ell$ and $c_2(S_m)=m(m+1)\ell^2/2$, where $\ell:=c_1(\OO_{\PP^2}(1))$.  Let $f:X\to\PP^2$ be the morphism constructed in \cref{sec:comparison}, with $f^*\OO(1)=H$.  Suppose that line bundles $N_1,N_2$ on $X$ satisfy $\mu_H(N_1)=\mu_H(N_2)$, $\chi(N_1)=\chi(N_2)$, and $c_i(N_1\oplus N_2)=c_i(f^*S_m)$ for $i=1,2$.  If $B_s$ represents $c_1(N_s)$ for $s=1,2$, the matrices $D_s=\Hmat^{-1/2}B_s\Hmat^{-1/2}$ satisfy
\[
 D_1+D_2=mI,\qquad \operatorname{tr}D_s=m,
 \qquad \det D_s=-\frac m2.
\]
Since $\det D_s=-m/2<0$, each Hermitian matrix $D_s$ has one positive and one negative eigenvalue.  The trace and determinant identities above give
\[
 \operatorname{tr}(D_s^2)=m^2+m,
 \qquad
 \operatorname{tr}\!\left(D_s\bigl((m+3/4)I-D_s\bigr)\right)=-\frac m4.
\]
For $m=5$, the line bundles $L_1,L_2$ chosen above satisfy these conditions, with $B_s=A_s$, and the displayed identities give the constants $23/4$ and $-5/4$ below.  The Gaussian-integer off-diagonal entry $1+13i$ realizes the determinant condition: with diagonal entries $8$ and $15$, the determinant condition requires $|1+13i|^2=170=1^2+13^2$.  Finally, the difference class $c_1(L_1)-c_1(L_2)$ has square $-1400$ and gives the $700$-dimensional extension space $\operatorname{Ext}^1(L_2,L_1)=H^1(X,L_1\otimes L_2^{-1})$.  A nonzero class supplies the nonsplit extensions used in \cref{sec:bridge}.  Thus the same lattice data produce both the curvature obstruction and the extensions.

Equip $X$ with the flat K\"ahler form and its volume form
\[
 \omega_H=\frac{i}{2}\,dz^{t}\Hmat\,d\bar z,
 \qquad dV_H=\frac{\omega_H^2}{2}.
\]
The corresponding tangent Hermitian product is $\langle u,v\rangle_H=v^{\dagger}\Hmat u$ in the coordinates $z$.  Take $dV=dV_H$ in the definition \eqref{eq:I} of $I_J$; this volume form is translation invariant.  Write $F_s=L_s^{-1}$ for $s=1,2$, so that $F_0=F_1\oplus F_2$.  For a Hermitian metric $g$ on a line bundle $L$, our curvature convention gives
\[
 c_1(L)=\left[\frac{i}{2\pi}R^g\right].
\]

We now express the curvature $R^{L_M}=\pi\,dz^tM^t\,d\bar z$ computed above as a tangent endomorphism.  The matrix $M^t$ is the coefficient matrix of $R^{L_M}/\pi$.  By contrast, the Hermitian endomorphism $C_M$ representing the real quadratic form $(2/i)\omega_M(u,\bar u)$ satisfies
\[
 \langle C_Mu,u\rangle_H
 =\frac{2}{i}\omega_M(u,\bar u)
 =\frac1\pi R^{L_M}(u,\bar u).
\]
In the coordinate frame $(\partial/\partial z_j)$ its matrix is $\Hmat^{-1}M$.  In the $H$-unitary frame obtained by multiplying that frame by $\Hmat^{-1/2}$, its matrix is
\[
 C_M=\Hmat^{-1/2}M\Hmat^{-1/2}.
\]
Indeed, the transpose in the form coefficient is precisely what converts $u^tM^t\bar u$ into the real Hermitian quadratic expression $u^{\dagger}Mu$.  We use the same symbol for this endomorphism and its matrix in that frame.  For $L_s$ we therefore write
\[
 C_s=\Hmat^{-1/2}A_s\Hmat^{-1/2}.
\]
Thus $R^{L_s}(u,\bar u)=\pi\langle C_su,u\rangle_H$ and $R^{F_s}(u,\bar u)=-\pi\langle C_su,u\rangle_H$ for $F_s=L_s^{-1}$.

We seek positive endomorphisms $Q_s$ satisfying $\operatorname{tr}(C_sQ_s)<0$: the weights must favour the negative curvature directions.  The following calculation gives such a choice.  Since $A_1+A_2=5\Hmat$, one has
\[
 C_1+C_2=5I.
\]
Write
\[
 C_1=\lambda I+K,\qquad C_2=\lambda I-K,
 \qquad \lambda=\frac52.
\]
Then the trace-free Hermitian endomorphism $K$ satisfies
\[
 K^2=d^2I,\qquad d=\frac{\sqrt{35}}2.
\]
Thus $d>\lambda$: the trace-free parts $K$ and $-K$ have eigenvalues $d$ and $-d$, so both $C_s$ are indefinite despite their positive traces.  If $d<\tau<d^2/\lambda$, then
\[
 Q_1=\tau I-K,\qquad Q_2=\tau I+K
\]
are positive definite and
\[
 \operatorname{tr}(C_sQ_s)=2(\lambda\tau-d^2)<0,\qquad s=1,2.
\]
With $P_s$ denoting the orthogonal projection onto $F_s$ for the split metric specified below, the tensor $Q_1\otimes P_1+Q_2\otimes P_2$ below pairs $C_s$ with $Q_s$ on the line factor $F_s$, so each summand contributes the negative trace just computed.

For the present matrices we take $\tau=13/4$, equivalently
\begin{equation}\label{eq:Qs}
 Q_s=\frac{23}{4}I-C_s.
\end{equation}
The characteristic polynomial of $C_s$ is $t^2-5t-5/2$, so its eigenvalues are $(5\pm\sqrt{35})/2$ and the eigenvalues of $Q_s$ are $(13\pm2\sqrt{35})/4$, both positive.  Moreover,
\begin{equation}\label{eq:Qnumerics}
 \operatorname{tr}Q_s=\frac{13}{2},\qquad
 \det Q_s=\frac{29}{16},\qquad
 \operatorname{tr}(C_sQ_s)=-\frac54.
\end{equation}
The last identity also follows directly from $\operatorname{tr}C_s^2=30$, since $\operatorname{tr}(C_sQ_s)=115/4-30=-5/4$.
Let $k$ be the orthogonal sum of constant-curvature metrics on $F_1,F_2$. The orthogonal projections $P_s:F_0\to F_s$ are parallel for its Chern connection. The global separable tensor is
\begin{equation}\label{eq:Qsplit}
 Q_0=Q_1\otimes P_1+Q_2\otimes P_2.
\end{equation}
This is a smooth separable tensor which, as an endomorphism of $T^{1,0}X\otimes F_0$, is positive definite. To see separability directly, diagonalise each $Q_s$ in the flat tangent space. Its positive eigenvalues express $Q_s\otimes P_s$ as a positive linear combination of tensor products of rank-one orthogonal projections. The eigenvectors may be chosen invariant on the torus; no trivialisation of the line bundles is needed because $P_s$ is global.

\begin{proposition}\label{prop:split-adjoint}
For the tensor $Q_0$ in \eqref{eq:Qsplit} and the split holomorphic structure $J_0$ on $F_0$, the functional $I_{J_0}$ in \eqref{eq:I} satisfies
\begin{equation}\label{eq:negative-adjoint}
 I_{J_0}(h')=-\frac{5\pi}{4}\int_X\tr(k^{-1}h')\,dV_H
\end{equation}
for every smooth Hermitian form $h'$. Thus $S_{J_0}=\mathcal L_{J_0}^*Q_0=-(5\pi/4)\Id_{F_0}$.
\end{proposition}

\begin{proof}
Write $a_s=h'|_{F_s}/k|_{F_s}$. The ratio is a globally defined smooth real function, even though a line frame is only local. The Chern connection preserves the splitting $F_0=F_1\oplus F_2$.  In a local holomorphic frame of $F_s$, write $k_s$ for the coefficient of $k|_{F_s}$.  For $v\in F_s$, formula \eqref{eq:lift} becomes
\[
 u^{\mathrm h,J_0}_v=(u,w),\qquad
 w=-k_s^{-1}(\partial_u k_s)v,
 \qquad w+k_s^{-1}(\partial_u k_s)v=0.
\]
In particular, $u^{\mathrm h,J_0}_v\in T_v^{1,0}F_s$. The Levi form on that tangent vector is consequently the Levi form of the restriction of $g_{h'}$ to $F_s$. Off-diagonal coefficients of $h'$ do not contribute.

For a smooth real function $a$ on the flat torus, let $\operatorname{Hess}_H(a)$ be the Hermitian endomorphism of $T^{1,0}X$ characterised by
\[
 \langle \operatorname{Hess}_H(a)u,u\rangle_H
 =\partial_u\partial_{\bar u}a.
\]
Its coordinate matrix is $\Hmat^{-1}(\partial_i\partial_{\bar j}a)^t$; in the $H$-unitary frame used above, it is
\[
 \Hmat^{-1/2}(\partial_i\partial_{\bar j}a)^t\Hmat^{-1/2}.
\]
The transpose distinguishes the coefficient matrix of $\partial\bar\partial a$ from the tangent endomorphism.  In a constant $\langle\cdot,\cdot\rangle_H$-orthonormal eigenbasis of $Q_s$, \cref{eq:scalar} gives
\[
 I_{J_0}(h')
 =\sum_{s=1}^2\int_X
 \left(\operatorname{tr}\bigl(Q_s\operatorname{Hess}_H(a_s)\bigr)
       +\pi\operatorname{tr}(C_sQ_s)a_s\right)dV_H.
\]
The eigenvectors of $Q_s$ are translation invariant and $dV_H$ is translation invariant.  Hence each directional mixed derivative integrates to zero, and therefore
\[
 \int_X\operatorname{tr}\bigl(Q_s\operatorname{Hess}_H(a_s)\bigr)dV_H=0.
\]
Applying \cref{eq:Qnumerics} and using $a_1+a_2=\tr(k^{-1}h')$ in the orthogonal splitting proves \cref{eq:negative-adjoint}.  In particular, the background test $h'=k$ gives $I_{J_0}(k)=-(5\pi/2)\operatorname{vol}_H(X)<0$, where $\operatorname{vol}_H(X)=\int_XdV_H$ is the volume of $X$.
\end{proof}

\begin{corollary}\label{cor:adjoint-open}
For the split holomorphic structure $J_0$ on $F_0$, the Hermitian metric $k$ constructed above, and the positive separable Hermitian tensor $Q_0$ in \eqref{eq:Qsplit}, there is a $C^2$-neighbourhood $\mathcal U$ of $J_0$ in the space of holomorphic structures on the fixed smooth bundle underlying $F_0$ such that
\[
 S_J=\mathcal L_J^*Q_0\le-\frac{5\pi}{8}\Id_{F_0}\qquad(J\in\mathcal U).
\]
Consequently, for every $J\in\mathcal U$, the dual bundle $(F_0,J)^*$ admits no smooth Griffiths-semipositive Hermitian metric.
\end{corollary}

\begin{proof}
By \cref{prop:split-adjoint}, $S_{J_0}=-(5\pi/4)\Id_{F_0}$. Apply \cref{prop:persistence} and shrink the $C^2$-neighbourhood of $J_0$ so that $\norm{S_J-S_{J_0}}_{C^0}<5\pi/8$.  This gives the asserted pointwise bound.  Proposition~\ref{crit:fixed} then excludes Griffiths-seminegative metrics on $(F_0,J)$, equivalently Griffiths-semipositive metrics on the dual bundle.
\end{proof}

\begin{remark}
Corollary~\ref{cor:adjoint-open} concerns Hermitian squared norms $G(z,v)=v^{\dagger}h(z)v$.  A strongly pseudoconvex complex Finsler squared norm on a holomorphic bundle $F\to X$ is a continuous function $G:F\to[0,\infty)$, positive and smooth away from the zero section, satisfying
\[
 G(z,\lambda v)=|\lambda|^2G(z,v),\qquad
 \Bigl(\frac{\partial^2G}{\partial v_a\partial\bar v_b}\Bigr)_{a,b}>0
 \quad(v\ne0),
\]
where $\lambda\in\C$ and $(v_a)$ are linear fibre coordinates.  This homogeneity does not require $G$ to have the form $v^{\dagger}h(z)v$.  The operator $\mathcal L_J$ acts on Hermitian forms $h=kA$, so the argument proving the corollary does not apply to general Finsler squared norms.  Negative Kobayashi curvature means that the metric induced by $G$ on $\OO_{\PP(F)}(-1)$ has negative Chern curvature \cite[\S3, equation~(3.19), and Theorem~4.1]{KobayashiFinsler}.  Kobayashi's ampleness criterion \cite[Theorem~5.1]{KobayashiFinsler} states that such a Finsler squared norm with negative curvature exists if and only if $F^*$ is ample.
\end{remark}
\section{An ample stable comparison bundle}\label{sec:comparison}

Dolgachev and Kapranov used Steiner bundles to study logarithmic differential forms along hyperplane arrangements \cite{DolgachevKapranov}.  A Steiner bundle is encoded by a matrix of linear forms, and its resolution makes the rank and Chern classes explicit.  This gives a source for our comparison bundle: the resolution below, with source rank five and target rank seven, supplies the numerical data used in the split construction, and a suitable matrix yields ampleness and stability.  Pullback to $X$ then gives a comparison bundle with the same Chern classes as the split bundle.

We choose the matrix so that two rank conditions hold: one makes its cokernel locally free, and the other makes the tautological morphism finite, which yields ampleness.  Let $V=\C^5$ and $U=\C^3$, and write $\Gr(k,W)$ for the Grassmannian of $k$-dimensional linear subspaces of a vector space $W$.  Then $\Gr(7,V\otimes U)$ has dimension $56$. The Segre variety
\[
 \PP(V)\times\PP(U)\subset\PP(V\otimes U)
\]
parametrises lines spanned by nonzero decomposable tensors $a\otimes z$ with $a\in V$ and $z\in U$, and has dimension $6$. The locus of $7$-planes containing a line spanned by a nonzero decomposable tensor has dimension at most
\[
 6+\dim\Gr(6,\C^{14})=54.
\]
Using the standard bases, identify $V\otimes U$ with $5\times3$ matrices and use the complex bilinear pairing $(x,y)\mapsto\sum_{i=1}^5\sum_{j=1}^3x_{ij}y_{ij}$.  Orthogonal complements in the construction of $S$ refer to this pairing, without complex conjugation.  The locus of $7$-planes whose orthogonal complement for this pairing contains a line spanned by a nonzero decomposable tensor has dimension at most
\[
 6+\dim\Gr(7,\C^{14})=55.
\]
The two loci just described are images of projective incidence varieties and are therefore closed.  Since their dimensions are strictly smaller than $56$, their union is a proper closed subset of the Grassmannian.  Hence there is a $7$-dimensional subspace $W\subset V\otimes U$ such that neither $W$ nor its orthogonal complement $W^\perp$ contains a nonzero decomposable tensor.

Identify $\PP^2$ with $\PP(U)$.  Choose a basis $w_1,\ldots,w_7$ of $W$, regarded as $5\times3$ matrices.  For $z=(z_1,z_2,z_3)\in U$, the $7\times5$ matrix $T(z)$ has entries
\[
 T(z)_{ai}=\sum_{j=1}^3(w_a)_{ij}z_j,\qquad 1\le a\le7,\quad1\le i\le5.
\]
This gives a sheaf map $\OO_{\PP^2}(-1)^5\to\OO_{\PP^2}^7$.

\begin{proposition}\label{prop:Steiner}
The cokernel $S$ in the sequence
\begin{equation}\label{eq:Steiner}
 0\longrightarrow\OO_{\PP^2}(-1)^5
 \xrightarrow{\ T\ }\OO_{\PP^2}^7
 \longrightarrow S\longrightarrow0
\end{equation}
is an ample $\mu_{\OO(1)}$-stable vector bundle of rank two with
\[
 c_1(S)=5\ell,\qquad c_2(S)=15\ell^2,
 \qquad \ell:=c_1\bigl(\OO_{\PP^2}(1)\bigr).
\]
\end{proposition}

\begin{proof}
If $T(z)a=0$ for $z\in U\setminus\{0\}$ and $a\in V\setminus\{0\}$, then the nonzero tensor $a\otimes z$ belongs to $W^\perp$, contrary to its construction. Thus $T(z)$ has rank five at every point, so $S$ is locally free of rank two.

The dual bundle $S^*$ is a subbundle of $\OO^7$. For $v=(v_1,\ldots,v_7)\in\C^7\setminus\{0\}$, write
\[
 M_v=\sum_{a=1}^7v_aw_a\in W.
\]
A point $([z],[v])\in\PP(S^*)$ satisfies $M_vz=0$. Since $W$ contains no nonzero rank-one tensor, $\operatorname{rank}M_v\ge2$; the fibre condition forces $\operatorname{rank}M_v=2$, so $\ker M_v=\C z$. Hence the tautological morphism
\[
 q:\PP(S^*)\longrightarrow\PP^6,\qquad([z],[v])\longmapsto[v],
\]
is quasi-finite.  Since it is projective, it is proper and hence finite.  Moreover,
\[
 q^*\OO_{\PP^6}(1)=\OO_{\PP(S^*)}(1).
\]
The restriction of $\OO_{\PP^6}(1)$ to the closed image of $q$ is ample, and its pullback by the finite surjection onto that image is ample by \cite[Propositions~4.1 and~4.3]{Hartshorne}.  Hence $S$ is ample.

For the total Chern class $c(S)=1+c_1(S)+c_2(S)$, the resolution gives
\[
 c(S)=(1-\ell)^{-5}=1+5\ell+15\ell^2.
\]
The resolution \eqref{eq:Steiner} makes $S$ a rank-two Steiner bundle on $\PP^2$. Its stability is a special case of \cite[Theorem~2.7]{BohnhorstSpindler}; see also \cite[Theorem~3.11]{DolgachevKapranov}. In the present case there is a short direct proof. For every integer $d\ge3$, twisting \eqref{eq:Steiner} by $\OO(-d)$ and using $H^0(\PP^2,\OO(-d))=H^1(\PP^2,\OO(-d-1))=0$ gives
\[
 H^0(\PP^2,S(-d))=0.
\]
If a rank-one subsheaf $G\subset S$ had slope at least $\mu_{\OO(1)}(S)=5/2$, its double dual would be a line bundle $G^{**}\simeq\OO(d)$ with $d\ge3$. Taking the double dual of the inclusion gives a nonzero map $\OO(d)\to S$, hence a nonzero section of $S(-d)$, contradicting the vanishing above. Thus $S$ is slope-stable.
\end{proof}

To obtain the comparison bundle on $X$, we pull $S$ back by a finite morphism to $\PP^2$.  Since $H$ is very ample, two general divisors in $|H|$ meet properly in a finite set.  Choose sections $s_0,s_1\in H^0(X,H)$ defining such divisors, and then choose $s_2$ nonvanishing on their intersection.  The three sections have no common zero and give a morphism
\[
 f:X\longrightarrow\PP^2,\qquad f=[s_0:s_1:s_2],
\]
with $f^*\OO_{\PP^2}(1)=H$. A positive-dimensional fibre would contain a curve on which $H$ has degree zero, contradicting ampleness; hence $f$ is quasi-finite.  Since $X$ is projective, $f$ is proper and therefore finite.  Its image is a closed irreducible surface in $\PP^2$, hence all of $\PP^2$, so $f$ is finite and surjective. Its degree is
\begin{equation}\label{eq:f-data}
 \deg f=h^2=40.
\end{equation}

The comparison bundle is
\[
 E_S:=f^*S.
\]
The projective bundle $\PP((f^*S)^*)$ is the base change of $\PP(S^*)$, and its tautological line bundle is the pullback of $\OO_{\PP(S^*)}(1)$.  Since this base change is finite and surjective, $E_S=f^*S$ is ample by \cite[Proposition~4.3]{Hartshorne}.  Moreover, \cref{prop:Steiner,eq:f-data} gives
\[
 c_1(E_S)=5h,\qquad\int_Xc_2(E_S)=15\deg f=600.
\]
Since $X$ is an abelian surface,
\[
 \chi(E_S)=\int_X\left(\frac12c_1(E_S)^2-c_2(E_S)\right)=500-600=-100,
\]
and hence
\[
 v(E_S)=(2,5h,-100)=\mathbf v=v(E_0).
\]

We now prove that $E_S=f^*S$ is $\mu_H$-stable by applying the pullback theorem of Biswas--Das--Parameswaran \cite[Theorem~1.2]{BDP} to the $\mu_{\OO_{\PP^2}(1)}$-stable bundle $S$ of \cref{prop:Steiner}.  The morphism $f$ is finite, separable, and surjective; both source and target are smooth projective varieties; and the induced map on the \emph{\'etale} fundamental groups is surjective because $\pi_1^{\mathrm{et}}(\PP^2)$ is trivial.  The polarisation on the source is precisely $f^*\OO_{\PP^2}(1)=H$.  Biswas--Das--Parameswaran's theorem \cite[Theorem~1.2]{BDP} therefore implies that $E_S$ is $\mu_H$-stable.

The bundle $E_S$ also admits a smooth Griffiths-semipositive Hermitian metric, namely the quotient metric induced by the standard flat metric on $\OO_X^7$ through the surjection $\OO_X^7\to E_S$ obtained from \eqref{eq:Steiner} \cite[Ch.~VII, Proposition~(6.10)]{DemaillyCADG}. Thus $E_S$ cannot serve as the bundle $E$ in \cref{thm:main}.

\section{Stable deformations and the ample locus}\label{sec:bridge}

In this section, we deform $E_0=L_1\oplus L_2$ to nonsplit extensions of $L_2$ by $L_1$.  We choose a nearby polarisation for which these extensions are stable and scale the extension class so that the dual holomorphic structures lie in the neighbourhood of \cref{cor:adjoint-open}.  We then place one such extension and the ample bundle $E_S$ in a common irreducible moduli space, where the density of the ample locus completes the proof of \cref{thm:main}.

We first determine the extension space and the polarisations for which the extensions are stable.  Since $\mu_H(L_1)=\mu_H(L_2)$, the subbundle $L_1$ prevents any extension of $L_2$ by $L_1$ from being $\mu_H$-stable.  We perturb the polarisation so that $L_1$ has smaller slope than $L_2$ and prove stability by estimating the slopes of all rank-one subsheaves.

For $\mathcal N=L_1\otimes L_2^{-1}$ and $\xi=c_1(\mathcal N)$, the matrices in \eqref{eq:split-matrices} and the intersection formula \eqref{eq:mixed-intersection} give
\[
 \xi\cdot h=0,\qquad \xi^2=-1400.
\]
Neither $\mathcal N$ nor $\mathcal N^{-1}$ has a nonzero section: a nonempty effective zero divisor would have positive intersection with $h$, while a nowhere-vanishing section would trivialise a line bundle with nonzero first Chern class.  Since $K_X\simeq\OO_X$, Serre duality also gives $H^2(X,\mathcal N)=H^2(X,\mathcal N^{-1})=0$.  Riemann--Roch \cite[Theorem~3.6.3]{BirkenhakeLange} therefore yields
\begin{equation}\label{eq:extension-cohomology}
 \dim_\C H^1(X,\mathcal N)=\dim_\C H^1(X,\mathcal N^{-1})=700.
\end{equation}

Let $N^1(X)_\R$ be the real vector space of divisor classes modulo numerical equivalence.  For a real ample class $a$, the slope of a torsion-free sheaf $G$ of positive rank is $\mu_a(G)=c_1(G)\cdot a/\rk G$.  Consider the open set
\begin{equation}\label{eq:extension-polarisations}
 \mathcal W=\left\{a\in N^1(X)_\R:
       a-\tfrac12h\text{ is ample},\quad 0<-\xi\cdot a<1\right\}.
\end{equation}
Every class in $\mathcal W$ is ample.  Moreover, $\mathcal W$ contains classes arbitrarily close to $h$: for sufficiently small $\varepsilon>0$, take the class $a_\varepsilon=h+\varepsilon\xi$.  Then $a_\varepsilon-h/2$ is ample and $-\xi\cdot a_\varepsilon=1400\varepsilon\in(0,1)$.

\begin{proposition}\label{prop:extension-stability}
For every $a\in\mathcal W$, every nonsplit extension
\begin{equation}\label{eq:nonsplit-extension}
 0\longrightarrow L_1\longrightarrow E\longrightarrow L_2\longrightarrow0
\end{equation}
is $\mu_a$-stable.
\end{proposition}

\begin{proof}
We adapt the line-subbundle argument in \cite[Example~1.2.10]{HuybrechtsLehn} to the present surface.  For a rank-one subsheaf $G\subset E$, its double dual $M=G^{**}$ is a line bundle with $c_1(M)=c_1(G)$, and the inclusion induces a nonzero map $M\to E$, by \cite[Tags~0B3N and~0AVC]{Stacks}. Write $\delta=-\xi\cdot a\in(0,1)$. If $M\to E\to L_2$ vanishes, then $M$ maps into $L_1$, so
\[
 \mu_a(G)-\mu_a(E)
 \le\mu_a(L_1)-\mu_a(E)=-\frac{\delta}{2}<0.
\]

Otherwise $M\simeq L_2(-D)$ for an effective divisor $D\ne0$, since $D=0$ would split the extension.  As $a-h/2$ is ample and $h$ is integral and ample, $D\cdot a>\tfrac12D\cdot h\ge\tfrac12$. Hence
\[
 \mu_a(G)-\mu_a(E)=\frac{\delta}{2}-D\cdot a
 <-\frac{1-\delta}{2}<0.
\]
Thus every rank-one subsheaf has smaller slope, proving stability.
\end{proof}

The following theorem gives stable bundles arbitrarily close to the split holomorphic structure.

\begin{theorem}\label{thm:stabledeform}
For every $a\in\mathcal W$, every integer $\nu\ge0$, and every $C^\nu$-neighbourhood of $\bar\partial_{E_0}$ on the fixed smooth bundle $E_0$, there exists an integrable Dolbeault operator in that neighbourhood whose holomorphic bundle $E_1$ is $\mu_a$-stable and satisfies
\[
 \det E_1\simeq H^{\otimes5},\qquad
 v(E_1)=(2,5h,-100),\qquad v(E_1)^2=1400.
\]
\end{theorem}

\begin{proof}
By \cref{eq:extension-cohomology}, choose a nonzero class $\theta\in H^1(X,\mathcal N)=\operatorname{Ext}^1(L_2,L_1)$ and a smooth $\bar\partial$-closed representative $\beta\in\Omega^{0,1}(\mathcal N)$.  On the fixed smooth bundle $E_0=L_1\oplus L_2$, consider the family, parametrised by $t\in\C$,
\begin{equation}\label{eq:extension-operators}
 \bar\partial_t=
 \begin{pmatrix}
  \bar\partial_{L_1}&t\beta\\
  0&\bar\partial_{L_2}
 \end{pmatrix}.
\end{equation}
The equation $\bar\partial\beta=0$ and the vanishing of the product of two strictly upper-triangular $2\times2$ matrices give $\bar\partial_t^2=0$.  The resulting holomorphic bundle $E^{(t)}$ is the extension with class $t\theta$ in \cref{eq:nonsplit-extension}; it is locally free since the two line bundles are locally free.  For $t\ne0$ it is nonsplit, and hence $\mu_a$-stable for every $a\in\mathcal W$.  Since $X$ is projective, these holomorphic bundles are algebraic by Serre's GAGA theorem \cite{SerreGAGA}.

The operator difference in \cref{eq:extension-operators} is $t\beta$, so it tends to zero in every fixed finite $C^\nu$ coefficient norm.  The same holds for the dual operators on $F_0=E_0^*$: their difference is the negative transpose of the matrix-valued form defining the difference on $E_0$.  The exact sequence gives $\det E^{(t)}=L_1\otimes L_2=H^{\otimes5}$ and preserves the Chern classes.
\end{proof}

The parameter $t$ supplies representatives close to the split operator. For $t\ne0$, the constant smooth bundle automorphism
\[
 g_t=\begin{pmatrix}t&0\\0&1\end{pmatrix},
 \qquad \bar\partial_t\circ g_t=g_t\circ\bar\partial_1,
\]
gives a holomorphic isomorphism $E^{(1)}\simeq E^{(t)}$. The matrix $g_t$ is singular at $t=0$, so these isomorphic bundles can have representatives converging to the split operator. These extensions are not ample: their quotient $L_2$ has $c_1(L_2)^2=-100$.  We now choose a common general polarisation for a stable extension near $E_0$ and the ample comparison bundle $E_S$.

\begin{lemma}\label{lem:polarisation-open}
Let $h=c_1(H)\in N^1(X)_\R$.  If a rank-two vector bundle $E$ on $X$ is $\mu_H$-stable, then it is $\mu_a$-stable for every $a$ in some neighbourhood of $h$ in the real ample cone.
\end{lemma}

\begin{proof}
Choose $m$ such that $E^*\otimes H^{\otimes m}$ is globally generated.  With $N=\dim H^0(X,E^*\otimes H^{\otimes m})$, the dual evaluation map gives an inclusion $E\hookrightarrow(H^{\otimes m})^{\oplus N}$.  For a rank-one subsheaf $G\subset E$, taking double duals as in the proof of \cref{prop:extension-stability} gives a line bundle $L=G^{**}$ mapping nontrivially to $E$ and having the same first Chern class as the original subsheaf.  Some component $L\to H^{\otimes m}$ is nonzero.  Thus $D=mh-c_1(L)$ is an effective integral numerical class, possibly zero, and, with $B=mh-\tfrac12c_1(E)$,
\[
 \mu_a(G)-\mu_a(E)=(B-D)\cdot a.
\]
Choose a basis $b_1,\ldots,b_\rho$ of $N^1(X)_\R$ and $\epsilon>0$ such that $h\pm\epsilon b_j$ are ample.  For every effective class $D$,
\[
 |D\cdot b_j|\le\epsilon^{-1}D\cdot h,\qquad
 D\cdot a\ge\tfrac12D\cdot h
 \quad\text{if }a=h+\sum_jt_jb_j,\quad\sum_j|t_j|<\epsilon/2.
\]
Take a bounded neighbourhood $U_0$ of $h$ in this region and $M>0$ with $|B\cdot a|<M$ on $U_0$.  If $D\cdot h>2M$, then $(B-D)\cdot a<0$ throughout $U_0$.  The remaining classes satisfy $D\cdot h\le2M$; the displayed bounds and nondegeneracy of the intersection pairing place them in a bounded subset of the lattice of integral numerical classes.  Only finitely many such classes occur.  For those arising from subsheaves of $E$, stability at $h$ gives $(B-D)\cdot h<0$.  Shrinking $U_0$ preserves these finitely many inequalities and proves the lemma.
\end{proof}

Apply \cref{lem:polarisation-open} to $E_S$, and let $\mathcal V$ be a neighbourhood of $h$ on which it remains stable.  The open set $\mathcal V\cap\mathcal W$ is nonempty.  The walls for $\mathbf v$ are locally finite \cite[p.~817]{Yoshioka}, so this intersection contains a rational ample class $\widehat h$ lying on no wall.  Choose an ample line bundle $H'$ with $c_1(H')=N\widehat h$ for a positive integer $N$.  This polarisation is $\mathbf v$-general, and positive rescaling preserves slope stability because $\mu_{N\widehat h}(G)=N\mu_{\widehat h}(G)$.

For a torsion-free sheaf $\mathcal F$ of positive rank, its reduced Hilbert polynomial with respect to $H'$ is
\[
 p_{H'}(\mathcal F,m)
 =\frac{\chi(\mathcal F\otimes H'^{\otimes m})}{\rk\mathcal F\,c_1(H')^2/2}.
\]
Gieseker stability requires $p_{H'}(\mathcal G,m)<p_{H'}(\mathcal F,m)$ for $m\gg0$ and every subsheaf $\mathcal G\subset\mathcal F$ of smaller positive rank.

Fix a sufficiently small nonzero $t_0$ in \cref{eq:extension-operators} so that the dual holomorphic structure of $E_1:=E^{(t_0)}$, under the fixed smooth identification $(E^{(t_0)})^*\simeq F_0$, belongs to the neighbourhood $\mathcal U$ of \cref{cor:adjoint-open}.  Both $E_1$ and $E_S$ are $\mu_{H'}$-stable, hence $H'$-Gieseker-stable \cite[Lemma~1.2.13]{HuybrechtsLehn}.  The background data $k$, $Q_0$, and $dV_H$ in the Levi obstruction continue to use $H$; only the stability polarisation changes to $H'$.

Let
\[
 \mathcal M=\mathcal M_{H'}(\mathbf v),\qquad \mathbf v=(2,5h,-100),
\]
be the moduli space of $H'$-Gieseker-stable sheaves on $X$ with this Mukai vector.  Only the Mukai vector is fixed: the determinant may vary in $\Pic^{5h}(X)$, the connected component of the Picard scheme parametrising line bundles with first Chern class $5h$.  The vector $\mathbf v$ is primitive, has positive rank and square $1400$, and $H'$ is $\mathbf v$-general.  Thus semistability equals stability for this vector \cite[\S1.2]{Yoshioka}.  Mukai's theorem \cite[Theorem~0.1]{Mukai1984} gives smoothness, and Yoshioka's theorem \cite[Theorem~0.1]{Yoshioka} gives a deformation equivalence
\[
 \mathcal M\sim_{\mathrm{def}}\widehat X\times\Hilb_X^{700}.
\]
Here $\widehat X=\Pic^0(X)$ is the dual abelian surface, where $\Pic^0(X)$ denotes the identity component of the Picard group, and $\Hilb_X^n$ is the Hilbert scheme of length-$n$ zero-dimensional subschemes of $X$.  The Hilbert scheme $\Hilb_X^{700}$ is connected \cite[Example~4.5.10]{HuybrechtsLehn}.  Consequently $\mathcal M$ is smooth and connected, hence irreducible, since distinct irreducible components of a smooth variety are disjoint.

\begin{lemma}\label{lem:local-families}
The locally free locus
\[
 \mathcal M^\circ=\{[E]\in\mathcal M:E\text{ is locally free}\}
\]
is Zariski open.  For every $[E]\in\mathcal M^\circ$, there are a pointed variety $(S,s)$, an \'etale morphism $\phi:S\to\mathcal M^\circ$ with $\phi(s)=[E]$, and a vector bundle $\mathcal E$ on $X\times S$ whose fibre over $y\in S$ represents $\phi(y)$.  Over a sufficiently small analytic neighbourhood of $s$, the dual fibres admit smooth identifications with $E^*$, equal to the identity at $s$, under which their Dolbeault operators vary continuously in every fixed finite $C^\nu$ coefficient norm.
\end{lemma}

\begin{proof}
By the existence of \'etale local universal families \cite[Corollary~4.3.5 and \S4.D, p.~139]{HuybrechtsLehn}, $\mathcal M$ admits an \'etale cover $\{\phi_\alpha:S_\alpha\to\mathcal M\}$ with flat families $\mathcal E_\alpha$ on $X\times S_\alpha$ whose fibre over each $y\in S_\alpha$ represents the moduli point $\phi_\alpha(y)$.  Here the standard construction is restricted to the determinant component $\Pic^{5h}(X)$; together with the Hilbert polynomial, this fixes $\mathbf v$ \cite[\S4.5, p.~113]{HuybrechtsLehn}.  For each $\alpha$, the set $S_\alpha^\circ$ of parameters for which the fibre is locally free on $X$ is Zariski open, by flatness of the family and properness of $X$.  The restriction of $\mathcal E_\alpha$ to $X\times S_\alpha^\circ$ is a vector bundle. Since \'etale morphisms are open,
\[
 \mathcal M^\circ=\bigcup_\alpha\phi_\alpha(S_\alpha^\circ)
\]
is Zariski open. For $[E]\in\mathcal M^\circ$, choose a point $s\in S_\alpha^\circ$ above $[E]$ and restrict this family to obtain the required $(S,s)$, $\phi$, and $\mathcal E$.

Since $\mathcal M$ is smooth, choose an analytic coordinate ball $B\subset S^{\mathrm{an}}$ about $s$ on which $\phi$ is biholomorphic onto its image.  Choose a smooth complex-linear connection on the dual family over $X\times B$.  Parallel transport along radial parameter segments, with $x\in X$ fixed, gives identifications with the central fibre $E^*$ that are the identity at $s$.  The pulled-back Dolbeault operators have coefficients smooth in $(x,y)$.  Compactness of $X$ makes this dependence continuous in every prescribed finite $C^\nu$ norm.
\end{proof}

\begin{proof}[Proof of \cref{thm:main}]
The locally free locus $\mathcal M^\circ$ contains $[E_1]$ and $[E_S]$, so it is nonempty. As an open subset of the irreducible variety $\mathcal M$, it is irreducible. Apply \cref{lem:local-families} at $[E_1]$, and write $y=0$ for the central parameter of its local family $\mathcal E$. Let $\tau_0:F_0\to E_1^*$ be the prescribed smooth identification, and let $P_y:E_1^*\to\mathcal E_y^*$ be the radial parallel transport used in the lemma. Set
\[
 \tau_y=P_y\circ\tau_0,\qquad
 \bar\partial_{J^{(y)}}:=\tau_y^{-1}\circ\bar\partial_{\mathcal E_y^*}\circ\tau_y.
\]
Since $P_0=\Id$, the identification at $y=0$ is $\tau_0$, so $J^{(0)}$ is precisely the dual structure of $E_1$ already chosen in $\mathcal U$. By $C^2$ continuity, after shrinking the analytic parameter neighbourhood, all the pulled-back dual structures $J^{(y)}$ lie in $\mathcal U$. With the same fixed background data, \cref{cor:adjoint-open} gives
\[
 S_{J^{(y)}}\le-\frac{5\pi}{8}\Id_{F_0}.
\]
Thus the image is an analytically open neighbourhood $U_{\mathrm{an}}$ of $[E_1]$ in $\mathcal M^\circ$.  No bundle represented in this neighbourhood admits a smooth Griffiths-semipositive Hermitian metric.

To prove that the ample locus is Zariski open, we use the algebraic \'etale families $\phi:S\to\mathcal M^\circ$ from \cref{lem:local-families} at all points.  For each such family $\mathcal E$, the morphism $\PP_{X\times S}(\mathcal E^*)\to S$ is proper, and its tautological line bundle restricts to $\OO_{\PP(E_y^*)}(1)$ on the fibre over $y$.  Openness of ampleness in proper families \cite[Theorem~1.2.17]{Lazarsfeld} shows that
\[
 S^+=\{y\in S:E_y\text{ is ample}\}
\]
is Zariski open.  An \'etale morphism is open in the Zariski topology, so each $\phi(S^+)$ is open.  Their union is precisely the ample locus
\[
 \mathcal M^+=\{[E]\in\mathcal M^\circ:E\text{ is ample}\}.
\]
This locus is nonempty because it contains $[E_S]$.  As a nonempty Zariski-open subset of the irreducible variety $\mathcal M^\circ$, it is dense in the analytic topology.  Thus $U_{\mathrm{an}}\cap\mathcal M^+$ is a nonempty analytically open subset of $\mathcal M^\circ$.  A bundle represented by any point in the intersection has rank two, is ample, and admits no smooth Griffiths-semipositive Hermitian metric.
\end{proof}

\section{An Open Question}\label{sec:cotangent}

The tangent and cotangent bundles are the most natural vector bundles on a smooth complex projective variety $Y$.  For the tangent bundle, Mori \cite{Mori1979} proved that ampleness characterises projective space (see also \cite{Peternell1996}).  Recently, Du--Guo--Xie \cite[Theorem~1.4]{DuGuoXie2026} gave a new proof based on the first analytic construction of rational curves on arbitrary Fano manifolds.  With the ample tangent case settled, it is natural to ask whether an ample cotangent bundle admits a smooth Griffiths-positive Hermitian metric.  Such a metric on $\Omega_Y^1$ is dual to a Hermitian metric on $T_Y$ with Griffiths-negative Chern curvature.

Xie \cite[Theorem~1.1]{Xie2018} proved that, for $n\ge2$ and $c\ge n$, a general complete intersection of $c$ hypersurfaces of sufficiently large degrees in $\PP^{n+c}$ has ample cotangent bundle.  His proof was inspired by Brotbek's work \cite{Brotbek2016} and uses the explicit symmetric differential forms constructed there.  Shortly afterwards, Brotbek--Darondeau \cite{BrotbekDarondeau2018} gave an alternative proof.

Under the stronger dimension condition $c\ge3n-1$, Mohsen \cite[Theorem~3(a)]{Mohsen2022} constructed smooth $n$-dimensional complete intersections in $\PP^{n+c}$ of sufficiently large equal degrees whose cotangent bundles carry Griffiths-positive Hermitian metrics.  This leads to the following question.

\begin{question}\label{qu:cotangent}
Fix integers $n\ge2$ and $c\ge n$.  For all sufficiently large degrees $d_1,\ldots,d_c$, does $\Omega_Y^1$ admit a smooth Griffiths-positive Hermitian metric when $Y\subset\PP^{n+c}$ is a general smooth complete intersection of multidegree $(d_1,\ldots,d_c)$?
\end{question}

In our construction, the holomorphic bundle varies over a fixed abelian surface, whose cotangent bundle is trivial.  In \cref{qu:cotangent}, the bundle instead varies with the underlying complete intersection.  The question therefore concerns how the intrinsic geometry of these varieties governs the existence of positive cotangent metrics.

\section*{Acknowledgements}
S.-Y. Xie acknowledges partial support from the National Key R\&D Program of China under Grants No.~2023YFA1010500 and No.~2021YFA1003100, and from the National Natural Science Foundation of China under Grants No.~12288201 and No.~12471081, as well as support from the Xiaomi Young Talents Program.

\section*{AI Use Disclosure}
The research direction and initial mathematical framework were developed by the authors. Generative AI was used interactively in the subsequent mathematical exploration, including the development, testing, and refinement of parts of the construction and proof. The authors independently verified the final arguments and take full responsibility for the paper.

\end{document}